\documentclass[11pt,a4paper,oneside,reqno]{amsart}

\usepackage[all]{xy}
\SelectTips{cm}{}  \calclayout
\usepackage{amsfonts,amssymb,amscd,amsmath,enumerate,verbatim,calc}

\usepackage{amscd,amssymb,amsopn,amsmath,amsthm,graphics,amsfonts,enumerate,verbatim,calc}
\usepackage[dvips]{graphicx}
\usepackage[colorlinks=true,linkcolor=blue,citecolor=blue]{hyperref}
\input xy
\xyoption{all}

\newcommand{\rt}{\rightarrow}
\newcommand{\lrt}{\longrightarrow}

\newcommand{\st}{\stackrel}
\newcommand{\pa}{\partial}

\newcommand{\C}{\mathbf{C} }

\newcommand{\K}{\mathbb{K} }
\newcommand{\X}{\mathbf{X} }
\newcommand{\Y}{\mathbf{Y} }
\newcommand{\T}{\mathbf{T} }
\newcommand{\I}{\mathbf{I} }

\newcommand{\F}{\mathbf{F} }

\newcommand{\CO}{\mathcal{O}}

\newcommand{\Z}{\mathbb{Z} }

\newcommand{\CE}{\mathcal{E}}
\newcommand{\CA}{\mathcal{A} }

\newcommand{\CB}{\mathcal{B} }
\newcommand{\CC}{\mathcal{C} }

\newcommand{\CF}{\mathcal{F} }
\newcommand{\CG}{\mathcal{G} }

\newcommand{\CS}{\mathcal{S} }
\newcommand{\CT}{\mathcal{T} }
\newcommand{\CX}{\mathcal{X} }

\newcommand{\CJ}{\mathcal{J} }

\newcommand{\Prj}{{\rm{Proj}}}
\newcommand{\Flat}{{\rm{Flat}}}

\newcommand{\KA}{\K(\CA)}

\newcommand{\DPFR}{{\mathbb{D}_{\rm{pur}}({\Flat} R)}}

\newcommand{\CPFA}{{\mathbb{C}_{\rm{pac}}({\Flat} \CA)}}
\newcommand{\DPFA}{{\mathbb{D}_{\rm{pac}}({\Flat} \CA)}}

\newcommand{\KPR}{{\mathbb{K}({\Prj}  R)}}

\newcommand{\KFA}{{\mathbb{K}({\Flat} \CA)}}

\newcommand{\KPFA}{{\mathbb{K}_{\rm{p}}({\Flat} \CA)}}

\newcommand{\im}{{\rm{Im}}}

\newcommand{\Ker}{{\rm{Ker}}}

\newcommand{\Hom}{{\rm{Hom}}}
\newcommand{\Ext}{{\rm{Ext}}}

\newcommand{\Lim}{\underset{\underset{i \in \mathrm{I}}{\longrightarrow}}{\lim}}

\newtheorem{theorem}{Theorem}[section]
\newtheorem{corollary}[theorem]{Corollary}
\newtheorem{lemma}[theorem]{Lemma}
\newtheorem{proposition}[theorem]{Proposition}

\theoremstyle{definition}
\newtheorem{definition}[theorem]{Definition}

\newtheorem{remark}[theorem]{Remark}
\newtheorem{s}[theorem]{}

\theoremstyle{plain}

\theoremstyle{definition}

\numberwithin{equation}{section}
\begin{document}

\title[The homotopy category of flat functors] {The homotopy  category of flat functors}
\author[E. Hosseini and A. Zaghian]{Esmaeil Hosseini and Ali Zaghian}

\address{Department of Mathematics, Shahid Chamran University of
Ahvaz, P.O.Box: 61357-83151, Ahvaz, Iran.}
\email{e.hosseini@scu.ac.ir}

\address{Department of Mathematics and Cryptography, Malek Ashtar University of
Technology, P.O.Box: 115-83145, Isfahan, Iran.}
\email{a.zaghian@mut-es.ac.ir}

\begin{abstract}
Let $\CC$ be a small category  and $\CG$ be a tensor Grothendieck
category. We define a notion of flatness in the category Fun$(\CC,
\CG)$ of all covariant functors from $\CC$ to $\CG$ and show that
the inclusion $\KFA\lrt\KA$ has a right adjoint where $\KA$ is the
homotopy category of $\CA$ and $\KFA$ its subcategory consisting of
 complexes of flat functors.  In addition, we find a replacement for the quotient
$\DPFA=\KFA/\KPFA$  of triangulated categories where $\KPFA$ is the
homotopy category of all pure acyclic complexes of flat functors.
\end{abstract}
\maketitle

\section{Introduction}
Functor categories are used as a potent tool for solving some
important problems in representation theory. They are  introduced by
Maurice Auslander in \cite{A66} and used in his proof of the first
Brauer-Thrall conjecture (see \cite{A78}, \cite{AR74}, \cite{AR72}).

In this work, we assume that $\CC$ is a small  category, $(\CG,
\textmd{-}\otimes\textmd{-})$ is a closed symmetric tensor
Grothendieck category and $\CA=\mathrm{Fun}(\CC, \CG)$ is the
category of all covariant functors from $\CC$ to $\CG$. The homotopy
category of $\CA$ has some interesting triangulated subcategories.
One of them is the homotopy category $\KFA$ of flat functors which
is studied by Amnon Neeman in \cite{Ne08} and \cite{Ne10}. He
invented that, when $\CC$ is a category with a single object and
$\CG$ is the category of modules over a ring $R$ with $1\neq 0$, the
homotopy category $\KPR$ of projective modules can be replaced by a
quotient of $\KFA$ modulo its thick subcategory consisting of all
pure acyclic complexes. This quotient  is called the pure derived
category of flat $R$-modules  and  denoted by $\DPFR$ (see
\cite{Mu07}, \cite{AS12}, \cite{HS13}, \cite{MS11} for more
details). This is an important future of Neeman's work, because,
there are some closed symmetric tensor Grothendieck categories with
no non-zero projective object (see \cite[III. Ex. 6.2]{Ha}).

We know that there are  different  methods to  define  flatness in
$\CG$. One of them is defined by the categorical purity  and the
other one is defined by the tensor  purity. But, categorical flats
are trivial in some situations. For example, if
$X=\mathbb{P}^n(R)=\mathrm{Proj}(R[x_1,...,x_n])$ is the projective
$n$-space over a commutative ring $R$, then the category
$\mathfrak{Qco}X$ of all quasi-coherent $\CO_X$-modules is a closed
symmetric tensor Grothendieck category. By \cite[Corollary
4.6]{ES15}, $\mathfrak{Qco}X$ does not have
 non-zero categorical flat objects. This shows that
categorical flats are rare in Fun$(\CC, \mathfrak{Qco}X)$. Due to
this lack of data the study needed to employ another way to define
flatness. This is motivated us to define another notion of flatness
in $\CA$.

This paper is organized as follows. In Section 2, we define our
concept of flatness in $\CA$ and show that any object in $\CA$ has a
flat precover and a cotorsion preenvelope. In Section 3, we show
that the category of all complexes in $\CA$ admits flat covers and
cotorsion envelopes. In Section 4,  we prove that the pure derived
category of flat functors can be replaced by the homotopy category
of dg-cotorsion complexes of flat functors.

Before starting, let us fix some notations and definitions.
\begin{s}{\sc Category of complexes.}
Let $\C(\CA)$ be the category of all complexes in $\CA$ (complexes
are write cohomologically). A complex in $\CA$ is called
\textit{acyclic} if all cohomological groups are trivial. Let
$\X=(X^i,\pa_\X^i)_{i\in\Z}$ be a complex in $\CA$, for any integer
$n$, $\X[n]$ denotes the complex $\X$ shifted $n$ degrees to the
left. The \textit{left} \textit{truncation} of $\X$ at $n$ is
denoted by $\X_{\leq n}$ and defined by the complex
$$\cdots \lrt \X^{n-2} \lrt \X^{n-1} \lrt \rm{Ker}\pa^{n}_\X \lrt
0.$$

Let  $f:\X\lrt \Y$ be a morphism of complexes.  We say that $f$ is a
\textit{quasi}-\textit{isomorphism} if for any $n\in\Z$,
$\mathrm{H}^n(f): \mathrm{H}^{n}(\X)\rt \mathrm{H}^{n}(\Y)$ of
cohomological  groups is an isomorphism. The \textit{mapping}
\textit{cone} of $f$ is denoted by $\mathrm{C}_f$ and defined by the
complex $(\X[1]\oplus\Y, \pa_{\mathrm{C}_f} )$ where
\[\pa_{\mathrm{C}_f}={\begin{pmatrix}\pa_{\X[1]}&0\\f[1]&\pa_{\Y}\end{pmatrix}.}\]
We know that a morphism $f:\X\lrt \Y$ of complexes is a
quasi-isomorphism if and only if $\mathrm{C}_f$ is acyclic.
\end{s}

\begin{s}{\sc Cotorsion theories.}
Let $\mathbb{A}$ be a Grothendieck category. The \textit{right}
\textit{orthogonal} of a class $\mathbb{S}$ in $\mathbb{A}$ is
defined by
$$\mathbb{S}^\perp:=\{\mathbf{B} \in \mathbb{A} \ | \ \Ext_{\mathbb{A}}^1(\mathbf{S},\mathbf{B})=0, \ {\rm{for \ all}}
\ \mathbf{S} \in \mathbb{S} \}.$$ The left orthogonal
${}^\perp\mathbb{S}$
 is defined dually. The pair $(\mathbb{X},
\mathbb{Y})$ of classes in  $\mathbb{A}$ is called a
\textit{cotorsion} \textit{theory} if $\mathbb{X}^\perp=\mathbb{Y}$
and $\mathbb{X}={}^\perp\mathbb{Y}$. An object $\mathbf{A} \in
\mathbb{A}$ has a \textit{special} $\mathbb{Y}$-\textit{preenvelope}
if there is an exact sequence $  0\rt \mathbf{A} \rt \Y' \rt \X'\rt0
$, where $\X'\in {}^\perp\mathbb{Y}$ and $\Y' \in \mathbb{Y}$. The
\textit{special} $\mathbb{X}$-\textit{precover} is defined dually. A
cotorsion theory $(\mathbb{X}, \mathbb{Y})$ in $\mathbb{A}$ is
called \textit{complete} if any object in $\mathbb{A}$ has a special
$\mathbb{X}$-precover and a special $\mathbb{Y}$-preenvelope.
\end{s}

\begin{s}{\sc Orthogonality in triangulated categories.}
Let $\CS$ be a thick  subcategory of a triangulated category $\CT$
and $ \CS^\perp=\{C \in \CT \ | \ \Hom_\CT(S,C)=0, \ {\rm{for \
all}} \ S \in \CS \}$. If we have a distinguished triangle
$\xymatrix@C-1pc{\mathbf{X}\ar[r]&\mathbf{C}\ar[r]&\mathbf{S}\ar[r]&\Sigma\mathbf{X}}$
such that $\C\in\CS^\perp$ and $\mathbf{S}\in\CS$ then, $\CS
\rightarrow \CT$ has a right adjoint (see  \cite[Lemma 3]{B}). So,
by \cite[Proposition 4.9.1]{K10}, we have a triangle equivalence
$\CS^\perp
 \lrt \CT/\CS$  of triangulated  categories.
\end{s}

\textbf{Setup:}  Let Flat$\CG$ be the class of all \textbf{tensor}
\textbf{flat} objects in $\CG$ and
$$\mathrm{Cot}\CG=\{\CC\in\CG|\forall ~~~\CF\in\mathrm{Flat}\CG,~~~~~
\Ext_{\CG}^1(\CF,\CC)=0 \}$$ be the class of all cotorsion objects.
In this paper, we assume that (Flat$\CG$, Cot$\CG$) is a
\textbf{complete} \textbf{cotorsion} \textbf{theory}.

\section{Flatness in functor categories}

In this section, we present our  notion of a flat functor and prove
that the corresponding cotorsion theory is complete.  Let
$\mathcal{H}om_{\CG}(\textmd{-},\textmd{-}):\CG\times\CG^{\mathrm{op}}\lrt\CG$
be the right adjoint of $\textmd{-}\otimes\textmd{-}$,  $\CJ$ be an
injective cogenerator in $\CG$ and
$(\textmd{-})^+=\mathcal{H}om_{\CG}(\textmd{-},\CJ)$. suppose that
$\CB=\mathrm{Fun}(\CC^{\mathrm{op}}, \CG)$ is the category of all
contravariant functors from $\CC^{\mathrm{op}}$ into $\CG$.
Therefore, $(\textmd{-})^+=\mathcal{H}om_{\CG}(\textmd{-},\CJ)$ is a
contravariant functor from $\CA$ to $\CB$. In the following
definition, we use $(\textmd{-})^+$ to define a notion of  purity in
$\CA$. This method was first used by Bo Stenstr\"{o}m in
\cite{Sten68}.

\begin{definition}\label{111}
A short exact sequence $\CE$ in $\CA$ is said to be pure if $\CE^+$
 splits.
\end{definition}

This definition  encouraged us to look at the class of all pure
injective functors. Recall that, a functor $\mathcal{P}$ is said to
be \emph{pure} \emph{injective} if it is injective with respect to
pure exact sequences in $\CA$. In the following result, we show that
the class of all pure injective functors is preenveloping. The proof
is similar to the proof of \cite[Proposition 2.7]{Ho13} which we
will prove it for further clarity.

\begin{proposition}\label{safi1}
The class of all pure injective functors is preenveloping.

\end{proposition}
\begin{proof}
Let $F$ be a functor. The canonical monomorphism $\varphi_F:F\lrt
F^{++}$ is pure, because  $F^{+}\lrt F^{+++}$ is the section of
$F^{+++}\lrt F^{+}$.  In addition, for a given diagram

\begin{align}\label{safi101}
\xymatrix@C-0.7pc@R-0.9pc{ 0 \ar[r]&G \ar[r]^i\ar[d]^f&K\\&F^{++}}
\end{align}we have the following commutative  diagram

\begin{align}
 \xymatrix@C-0.7pc@R-0.9pc{ & G \ar@{->}[rr]^i\ar@{-}[d]\ar[dd] & &
K \ar@{->}[dd]^h
\\
F^{++} \ar@{<-}[ur]^f\ar@{->}[dd]^s &
\\
& G^{++} \ar@{-}[r]^{i^{++}}\ar[rr] & & K^{++}
\\
F^{++++} \ar@{<-}[ur]^{f^{++}} & }
\end{align}
such that $\xymatrix@C-0.7pc@R-0.9pc{ s_1:K^{++} \ar[r] &  G^{++}}$
is the section of $i^{++}$ and  $ \xymatrix@C-0.7pc@R-0.9pc{ s_2:
F^{++++} \ar[r] & F^{++}}$ is the section of $s$. Then,
$s_2f^{++}s_1h: K\lrt F^{++}$ completes  \eqref{safi101} to a
commutative diagram. So,  $F^{++}$ is pure injective.

\end{proof}
\begin{corollary}\label{safi11}
For a given functor $F$,  $F^+$ is pure injective.
\end{corollary}
\begin{proof}
Since $F^+$  is a direct summand of $F^{+++}$ then by Proposition
\ref{safi1}, $F^+$ is pure injective.
\end{proof}

In \cite[Theorem 6]{Her}, the author used the categorical notion of
purity and proved that $\CA$ has enough categorical pure injective
objects. This class of pure injectives is different from objects
which are characterized in Proposition \ref{safi1}. Moreover, if
$\CE$ is an exact sequence in $\CA$ then  we deduce a degree-wise
pure exact sequence $0\to\CE\to\CE^{++}$ of complexes.  This is the
most important future of Definition \ref{111}.

Next, we use Definition \ref{111}  and extend the $\otimes$-purity
to $\CA$.

\begin{definition}\label{safi11}
A functor $F$ in $\CA$ is called flat if any short exact sequence
ending in $F$ is pure.
\end{definition}

It is necessary to say that, if $\CA$ has enough projective objects,
then Definition \ref{safi11} and the categorical notion of flatness
are equivalent (see \cite[Theorem 3]{Sten68}). But, in general case,
$\CA$ does not have non-zero categorical flat objects and Definition
\ref{safi11} is the only notion of flatness  in $\CA$.

In the next result, we prove the well-known relation between flat
and injective functors i.e. when Flat$\CA$ is the class of all flat
objects in $\CA$ and Inj$\CB$ is the class of all injective objects
in $\CB$, we show that
$(\textmd{-})^+:\mathrm{Flat}\CA\lrt\mathrm{Inj}\CB$ is a
contravariant functor (see the proof of \cite[Theorem 2.13]{Ho13}).

\begin{proposition}\label{p0}
A functor $F$ in $\CA$ is flat if and only if $F^+$ is  injective in
$\CB$.
\end{proposition}
\begin{proof}
If $F^+$ is injective, then  any  short exact sequence ending in $F$
is pure in $\CA$, because $\CE^+$ splits. So, $F$ is flat.

Conversely, assume that $F$ is  flat and

\begin{align}\label{mess0}
 \xymatrix@C-0.7pc@R-0.9pc{0\ar[r]& F^+\ar[r]& G\ar[r]& H\ar[r]& 0}
\end{align}be an arbitrary exact sequence in $\CB$. It is enough to show that
\eqref{mess0} splits in $\CB$. The top row of the following pullback
diagram
\[\xymatrix@C-0.7pc@R-0.9pc{0\ar[r]&H^+\ar[r]\ar@{=}[d]&P\ar[r]^h\ar[d]^g&F\ar[r]\ar[d]^i&0\\
0\ar[r]&H^+\ar[r]&G^+\ar[r]^{f^+}&F^{++}\ar[r]&0}\] is pure by
assumption. Hence, it is split ($H^+$ is pure injective). So, there
is a morphism $h':F\lrt P$ such that $hh'=1_{F}$. Therefore, if
$g_1=gh':F\lrt G^+$, then $f^+g_1=f^+gh'=ihh'=i$. Furthermore, the
following commutative diagram

\[\xymatrix@C-0.9pc@R-0.9pc{0\ar[r]&F^+\ar[r]^f\ar[d]^{j}&G\ar[r]\ar[d]^k&H\ar[r]\ar[d]&0\\
0\ar[r]&F^{+++}\ar[r]^{f^{++}}&G^{++}\ar[r]&H^{++}\ar[r]&0. }\]
implies  that  $g_1^+kf = g_1^+f^{++}j = i^+j = 1_{F^+}$. So,
\eqref{mess0} splits.

\end{proof}

Now, we need to introduce some notations. Let $f:c\lrt d$ be a
morphism in $\CC$. We write $s(f)=c$ and $t(f)=d$. A path in $\CC$
is a sequence of morphisms. For a given $c\in\CC$, the functor
$E_c:\CA\lrt\CG$ is defined by $E_c(F)=F(c)$. This is an exact
functor with an exact  right adjoint $S^c:\CG\lrt\CA$ defined by
$S^c(\CF)(e)=\prod_{\mathrm{Hom}_{\CC}(e,c)}\CF$  (see \cite[pp.
317]{EEG} for more details). It can be shown that, if $\CE$ is an
injective cogenerator for $\CG$ then $S^c(\CE)$ is an injective
cogenerator for $\CA$. We say that $\CC$ is \textit{left}
(\textit{right}) \textit{rooted} if there exists no path  of the
form $\cdots\to\bullet\to\bullet\to\bullet$
($\bullet\to\bullet\to\bullet\to\cdots$). Clearly,  $\CC$ is left
rooted if and only if $\CC^{\mathrm{op}}$ is right rooted.

\begin{proposition}\label{sim3}
Let $I$ be an injective object in $\CA$. Then,
\begin{itemize}
\item[i)] For any  $c\in\CC$, $I(c)$ is  injective in $\CG$.
\item [ii)] For any $c\in\CC$, the canonical morphism $I(c)\lrt\prod_{s(f)=c}I(t(f))$ is a split epimorphism.
\end{itemize}
\end{proposition}
\begin{proof}
Let $\CE$ be an injective cogenerator in $\CG$. Then, for any
$c\in\CC$, $S^c(\CE)$ satisfies in (i) and (ii). In addition, (i)
and (ii) are preserved under products and direct summands. Since
$S^c(\CE)$ (varying $c\in\CC$ and $\CE$) cogenerate $\CA$ and any
injective functor is a direct summand  of a  product of  various
$S^c(\CE)$. So, any injective object satisfies in $(i)$ and $(ii)$.
\end{proof}
The converse of this proposition holds if $\CC$ is right rooted (see
\cite[Theorem 4.2]{EEG}).
\begin{proposition}\label{sim2}
Let $F$ be a flat object in $\CA$. Then,
\begin{itemize}
\item[i)] For any  $c\in\CC$, $F(c)$ is  flat in $\CG$.
\item [ii)] For any $c\in\CC$, the canonical morphism $\coprod_{t(f)=c}F(s(f))\lrt F(c)$ is a pure
monomorphism.
\end{itemize}
\end{proposition}
\begin{proof}
We know that for a given flat functor $F$ in $\CA$,  $F^+$ is
injective in $\CB$. So, by Propositions \ref{p0} and \ref{sim3}, $F$
has the desired properties.
\end{proof}

Here, we introduce another adjoint pair of functors. For a given
$c\in\CC$, $E_c$ has an exact left adjoint $S_c:\CG\lrt\CA$ which is
defined by $S_c(\CF)(e)=\oplus_{\mathrm{Hom}_{\CC}(c,e)}\CF$ (see
\cite{Mit81} for more details).
\begin{proposition}\label{shoh1}
The category $\CA$ is  Grothendieck.
\end{proposition}
\begin{proof}
A sequence in $\CA$ is exact if it is exact in any $c\in\CC$. This
causes that $\CA$ is an abelian category. In addition,  coproducts
and direct limits may be computed componentwise and so they are
exact. If $\CS$ is a set of generators in $\CG$ then, $\{S_c(\CX)|
c\in\CC, \CX\in\mathcal{S}\}$ is a set of generators for $\CA$.
Therefore, $\CA$ is  Grothendieck.
\end{proof}

\begin{proposition}\label{p1}
Let  $\xymatrix@C-0.7pc@R-0.9pc{\CE:0\ar[r]& F\ar[r]& G\ar[r]&
K\ar[r]& 0}$ be a pure exact sequence in $\CA$. Then $G$ is flat if
and only if $F$, $K$ are flats.
\end{proposition}
\begin{proof}
By assumption,   $\CE^+$  splits. Then   $G^+$ is injective if and
only if $F^+$ and $K^+$ are injectives. Consequently,  by
Proposition \ref{p0}, $G$ is flat if and only if $F$ and $K$ are
flats.
\end{proof}

\begin{lemma}\label{p001}
Any direct limit of pure exact sequences is pure.
\end{lemma}
\begin{proof}
Let $(\xymatrix@C-0.7pc@R-0.9pc{\CE_i:0\ar[r]& F_i\ar[r]& G_i\ar[r]&
H_i\ar[r]& 0)_{i\in I}}$ be a direct system of pure exact sequences
in $\CA$. Then, $(\CE_i^+)_{i\in I}$  is an inverse system of split
exact sequences in $\CB$. So, $
(\Lim{\CE_i})^{+}\cong\underleftarrow{\mathrm{lim}}{\CE_i}^{+}$
splits. Therefore, $\Lim{\CE_i}$ is a pure exact sequence.
\end{proof}

\begin{remark}\label{Abtino}
It is known that in a $\lambda$-presentable Grothendieck  category
any $\lambda$-pure exact sequence is a direct limit of split exact
sequences (see \cite{AR94}, \cite{Cr94}) and hence by Lemma
\ref{p001}, it is pure in the sense of Definition \ref{111}.
Particularly, for any family $\{X_i\}_{i\in I}$ in $\CA$, the
categorical pure exact sequence
\begin{align}\label{n.3}
\xymatrix@C-0.7pc@R-0.9pc{ 0\ar[r]&K\ar[r]&\oplus_{i\in
I}X_i\ar[r]&\Lim X_i\ar[r]&0}
\end{align}
 is pure in the sense of Definition \ref{111}.
\end{remark}

\begin{proposition}\label{p2}
The class $\mathrm{Flat}\CA$ is closed under pure subobjects, pure
quotients and direct limits.
\end{proposition}
\begin{proof}
By Proposition \ref{p1}, $\mathrm{Flat}\CA$ is closed under pure
subobjects and pure quotients. If $\{X_i\}_{i\in I}$ is a family of
flat functors, then $(\oplus_{i\in I}X_i)^+=\prod_{i\in I} X_i^+$ is
injective in $\CB$ and so by Proposition \ref{p0}, $\oplus_{i\in
I}X_i$ is flat. Hence,  by Remark \ref{Abtino} and Proposition
\ref{p1},  $\mathrm{Flat}\CA$ is closed under direct limits.
\end{proof}
\begin{proposition}\label{shoh10}
The category $\CA$ has enough flat objects.
\end{proposition}
\begin{proof}
Let $c\in\CC$ and $\CF$ be a flat object in $\CG$. Then, we can
consider $S_c(\CF)$ as a functor over a left rooted category and
hence  by Proposition \ref{sim3}, $(S_c(\CF))^+$ is  injective. It
follows that, $S_c(\CF)$ is flat (Proposition \ref{p0}). For a given
$G\in\CA$ and $c\in\CC$, let $\CF(c)\to G(c)\to 0$ be a flat
precover of $G(c)$. Then, we have an exact sequence
$\oplus_{c\in\CC}S_c(\CF(c))\to G\to 0$ where
$\oplus_{c\in\CC}S_c(\CF(c))$ is  flat by Proposition \ref{p2}.
\end{proof}

\begin{proposition}\label{p3}
The pair $(\mathrm{Flat}\CA,\mathrm{Cot}\CA)$ is a complete
cotorsion theory in  $\CA$.
\end{proposition}
\begin{proof}
By Proposition  \ref{p2},  $\mathrm{Flat}\CA$ is closed under direct
limits. So, by \cite[Theorem 3.3]{E}, it is a covering class.
Combine this with Proposition \ref{shoh10} and deduce that any
object in $\CA$ has a special flat precover (\cite[Corollary
7.2.3.]{EJ}). Let $G$ be a given functor and
$\xymatrix@C-0.7pc@R-0.9pc{ 0\ar[r]&G\ar[r]&I\ar[r]&K\ar[r]&0}$ be
its injective envelope. Then the pullback diagram

\[\xymatrix@C-0.7pc@R-0.9pc{&&0\ar[d]&0\ar[d]\\&&C\ar[d]\ar@{=}[r]&C\ar[d]\\0\ar[r]&G\ar[r]\ar@{=}[d]&P\ar[r]\ar[d]&F\ar[r]\ar[d]&0\\
0\ar[r]&G\ar[r]&I\ar[r]\ar[d]&K\ar[r]\ar[d]&0\\
&&0&0}\] completes the proof where $\xymatrix@C-0.7pc@R-0.9pc{
0\ar[r]&C\ar[r]&F\ar[r]&K\ar[r]&0}$ is a special flat precover of
$K$.
\end{proof}

\vspace{0.2cm}

\section{Flat cotorsion theory in $\C(\CA)$}

This section is devoted to the study of  flat cotorsion theory in
$\C(\CA)$.  An acyclic complex $\cdots \rt X^{n-1}
\st{\pa^{n-1}}{\rt} X^n \st{\pa^{n}}{\rt} X^{n+1} \rt \cdots$ in
$\CA$  is called \textit{pure} \textit{acyclic} if for any $n\in\Z$,
the exact sequence
\[\xymatrix@C-0.7pc@R-0.9pc{0\ar[r]&\Ker\pa^{n}\ar[r]&X^n\ar[r]&\im\pa^{n}\ar[r]&0}\]
is pure. A complex $$\F:\cdots \rt F^{n-1} \st{\pa_\F^{n-1}}{\rt}
F^n \st{\pa_\F^{n}}{\rt} F^{n+1} \rt \cdots$$ in $\CA$ is called
\textit{flat} if it is  a pure acyclic complex of flats or
equivalently, it is an acyclic complex such that for any $n\in\Z$,
$\Ker\pa_\F^{n}$ is flat (Proposition \ref{p0}). Let $\CPFA$ be the
class  of all flat complexes in $\CA$. A complex $\C$ in $\CA$ is
called dg-cotorsion if $\C\in\CPFA^\perp$. An acyclic complex
$\C=(C^n,\partial_\C^n)$ of cotorsion objects in $\CA$ is called
cotorsion if for any $n\in\Z$, $\Ker\partial^n_\C$ is cotorsion.
Clearly, dg-cotorsion complexes are not necessarily cotorsion. But,
it can be shown that any cotorsion complex is dg-cotorsion.

\begin{proposition} \label{cotor1}
Let $\mathbf{G}=(G^i, \pa_\mathbf{G}^i)\in \CPFA^{\perp}$. Then, for
any $i\in \Z$, $G^i$ is a cotorsion functor.

\end{proposition}
\begin{proof}
For $i\in\Z$, assume that
$$\xymatrix@C-.5pc{0\ar[r]&G^i\ar[r]^f&C\ar[r]^g&F\ar[r]&0}$$ be
a cotorsion   preenvelope of $G^i$. By the  pushout of $f$ and
$\pa^i_\mathbf{G}$, we deduce the following exact sequence

\[\xymatrix@C-0.7pc@R-0.9pc{&\vdots\ar[d]&\vdots\ar[d]&\vdots\ar[d]&\\0\ar[r]&G^{i-1}\ar@{=}[r]\ar[d]^{\pa^{i-1}}
&G^{i-1}\ar[r]\ar[d]^{f\circ \pa^{i-1}}&0\ar[d]\ar[r]&0\\
0\ar[r]&G^i\ar[r]^f\ar[d]^{\pa^{i}}
&C\ar[r]\ar[d]^s&F\ar@{=}[d]\ar[r]&0
\\0\ar[r]&G^{i+1}\ar[r]\ar[d]
&G\ar[r]\ar[d]&F\ar[d]\ar[r]&0\\0\ar[r]&G^{i+2}\ar@{=}[r]\ar[d]
&G^{i+1}\ar[r]\ar[d]&0\ar[d]\ar[r]&0\\
&\vdots&\vdots&\vdots}\] of complexes  which is split by the choice
of $\mathbf{G}$ (the last column is a flat complex). Then, $G^i$ is
cotorsion.

\end{proof}

Let $\C(\mathrm{dg}\textmd{-}\mathrm{Cot}\CA)$ be the class of all
dg-cotorsion complexes in $\CA$.  We prove that the pair
$(\CPFA,\C(\mathrm{dg}\textmd{-}\mathrm{Cot}\CA))$ is a complete
cotorsion theory in $\C(\CA)$. First of all, we prove the assertion
in the category of all short exact sequences in $\CA$.

\begin{lemma}\label{107}
Any short exact sequence in $\CA$ has a special flat precover and a
special  cotorsion preenvelope.
\end{lemma}
\begin{proof}
Let $\xymatrix@C-0.9pc{\mathbf{G}:
0\ar[r]&G'\ar[r]&G\ar[r]&G''\ar[r]&0}$ be an exact sequence in
$\CA$. We find an exact sequence
$\xymatrix@C-0.7pc@R-0.9pc{0\ar[r]&\C\ar[r]&\F\ar[r]&\mathbf{G}\ar[r]&0}$
in $\C(\CA)$ where $\F$ is a short  exact  sequence of flat functors
and $\C$ is a short  exact sequence of cotorsion functors. Let
$$\xymatrix@C-0.7pc@R-0.9pc{0\ar[r]&C''\ar[r]&F''\ar[r]&G''\ar[r]&0}$$
be a special flat precover of $G''$. Consider the following pullback
diagram
\[\xymatrix@C-0.7pc@R-0.9pc{&&0\ar[d]&0\ar[d]\\&&C''\ar[d]\ar@{=}[r]&C''\ar[d]\\0\ar[r]&G'
\ar[r]^j\ar@{=}[d]&P\ar[r]\ar[d]&F''\ar[r]\ar[d]&0\\
0\ar[r]&G'\ar[r]&G\ar[r]\ar[d]&G''\ar[r]\ar[d]&0\\
&&0&0}\] and let
$\xymatrix@C-0.7pc@R-0.9pc{0\ar[r]&C'\ar[r]&F\ar[r]^h&P\ar[r]&0}$ be
a special flat precover  of $P$.  The  pullback of $j$ and $h$
completes the proof.
\end{proof}
\begin{proposition}\label{c107}
Any bounded above acyclic complex in $\CA$ has a special flat
precover and a special cotorsion preenvelope.
\end{proposition}
\begin{proof}
Let  $\Y$ be a bounded above acyclic complex in $\CA$. We construct
an exact sequence
$\xymatrix@C-0.7pc@R-0.9pc{0\ar[r]&\C\ar[r]&\F\ar[r]&\Y\ar[r]&0}$ in
$\C(\CA)$ where $\F$ is a bounded above flat complex and $\C$ is a
bounded above acyclic complex of cotorsions. Without loss of
generality, we can assume that $\Y=(Y^i, \pa_\Y^i)_{i\leq 0}$. By
Lemma \ref{107}, we have the following commutative  diagram
\[\xymatrix@C-1pc@R-1.2pc{&0\ar[d]&0\ar[d]&0\ar[d]&\\0\ar[r]&T^{-1}\ar[r]\ar[d]&C^{-1}\ar[r]\ar[d]&C^0\ar[r]\ar[d]&0\\0\ar[r]&K^{-1}\ar[r]\ar[d]&F^{-1}\ar[r]\ar[d]&F^0\ar[r]\ar[d]&0\\
0\ar[r]&\mathrm{Ker}(\pa^{-1})\ar[r]\ar[d]&Y^{-1}\ar[r]\ar[d]&Y^0\ar[r]\ar[d]&0\\&0&0&0&}\]
in $\CA$ where  the middle row is an exact sequence of flats and the
top row is an exact sequence of cotorsions. For any $i< 0$, we use
Lemma \ref{107} and an inductive procedure to obtain  pure exact
sequences
$\xymatrix@C-1pc{0\ar[r]&K^{i-1}\ar[r]&F^{i-1}\ar[r]&K^i\ar[r]&0}$
of flats and  exact sequences
$\xymatrix@C-1pc{0\ar[r]&T^{i-1}\ar[r]&C^{i-1}\ar[r]&T^i\ar[r]&0}$
of cotorsions. For any $i>0$, set $C^i:=0$, $F^i:=0$. Then,
$\mathbf{F}=(F^i)$ and $\mathbf{C}=(C^i)$ have the desired property.

\end{proof}

\begin{theorem}\label{2111}
Any acyclic complex in $\CA$ has a special flat precover and a
special cotorsion preenvelope.
\end{theorem}
\begin{proof}
Let $\X=(X^i, \pa_\X^i)$ be an acyclic  complex in $\CA$. We find an
exact sequence
$\xymatrix@C-0.7pc@R-0.9pc{0\ar[r]&\X\ar[r]&\C\ar[r]&\F\ar[r]&0}$ of
complexes, where $\F$ is flat and $\C$ is cotorsion. Let
$$\xymatrix@C-0.7pc@R-0.9pc{0\ar[r]&\X^{\leq
0}\ar[r]&\I^{\leq 0}\ar[r]^g&\T^{\leq 0}\ar[r]&0}$$ be the injective
envelope of $\X^{\leq 0}$ where $\I^{\leq 0}$ and $\T^{\leq 0}$ are
bounded above acyclic complexes.  By Lemma \ref{c107}, we have a
short exact sequence $\xymatrix@C-0.7pc@R-0.9pc{0\ar[r]&\C'
\ar[r]&\F'\ar[r]^f&\T\ar[r]&0}$ of complexes where $\F'$ is a flat
complex and $\C'$ is  cotorsion. By the pullback of $f$ and $g$, we
deduce  the  exact sequence
$\xymatrix@C-0.7pc@R-0.9pc{0\ar[r]&\X^{\leq 0}
\ar[r]&\C_1\ar[r]^f&\F'\ar[r]&0}$ where $\C_1$ is a bounded above
acyclic complex of cotorsion objects ($\forall i> 0, \C_1^i =0$) and
$\F'$ is a bounded above flat complex ($\forall i> 0, \F'^i =0$). We
use Lemma \ref{107} and an  inductive procedure to construct a
cotorsion complex $\C$ from $\C_1$ and a flat complex $\F$ from
$\F'$ which satisfy  in the following exact sequence
$$\xymatrix@C-0.7pc@R-0.9pc{0\ar[r]&\X\ar[r]&\C\ar[r]&\F\ar[r]&0}$$
of complexes.  Consequently, any acyclic complex in $\CA$ has a
special flat precover and a special cotorsion preenvelope.
\end{proof}

\begin{corollary}\label{c13}
The pair $(\CPFA, \C(\mathrm{dg}\textmd{-}\mathrm{Cot}\CA))$ is a
complete cotorsion theory in $\C(\CA)$.
\end{corollary}
\begin{proof}
Let $\X$ be a complex in $\CA$. There is a quasi-isomorphism $f:
\X[-1]\lrt\I$ where $\I$ is a dg-injective complex in $\CA$
(\cite{Sp88}). By Theorem \ref{2111}, there is an exact sequence
$\xymatrix@C-0.7pc@R-0.9pc{0\ar[r]&\C'\ar[r]&\F\ar[r]&\mathrm{Cone}(f)\ar[r]&0}$
of complexes where $\F$ is flat and $\C$ is cotorsion. Then, the
pullback of $\xymatrix@C-0.7pc{\F\ar[r]&\mathrm{Cone}(f)}$ and
$\xymatrix@C-0.7pc{\I\ar[r]&\mathrm{Cone}(f)}$ completes  the proof.
\end{proof}

\section{The Pure derived category of of flats}
Let $\K(\CA)$ be the homotopy category of $\CA$ and $\KFA$ be its
full subcategory consisting of all complexes of flat functors. Let
$\KPFA$ be the full subcategory of $\KFA$ consisting of all flat
complexes,  $\mathbb{K}({\rm dg}$-${\rm Cof}\CA)$ be the essential
image of dg-cotorsion complexes of flats in $\CA$. Indeed,
$\mathbb{K}({\rm dg}$-${\rm Cof}\CA)$ is a full triangulated
subcategory of $\KFA$ and it is closed under isomorphisms. In this
section, we prove that $\mathbb{K}({\rm dg}$-${\rm Cof}\CA)$ and the
pure derived category $\DPFA$ of flats in $\CA$ are equivalent as
triangulated categories.

\begin{theorem}\label{c13440}
The inclusion  $\KFA\lrt\KA$ has a right adjoint.
\end{theorem}
\begin{proof}
Let $\alpha$ be an infinite cardinal. By Proposition \ref{p2},
$\mathrm{Flat}\CA$ is closed under $\alpha$-direct limits and so it
is closed under $\alpha$-pure subobject. So, by \cite[Theorem
5]{K12}, the inclusion $\KFA\lrt\K(\CA)$ has a right adjoint.
\end{proof}
Let $\KPFA$ be the thick subcategory of $\KFA$ consisting of pure
acyclic complexes and $\DPFA=\KFA/\KPFA$ be the pure derived
category of flat functors. In the following, we prove the main
result of this work.
\begin{theorem}\label{c1344}
There is an equivalence   $\mathbb{K}({\rm dg}$-${\rm
Cof}\CA)\lrt\DPFA$ of triangulated categories.
\end{theorem}
\begin{proof}
The proof contains two parts. In part one, we show that
$\KPFA^\perp$ and $\DPFA$ are equivalent. Let $\X$ be a complex of
flats. By Theorem \ref{2111}, there is an exact sequence
\begin{align}\label{connn}
\xymatrix@C-0.7pc@R-0.9pc{0\ar[r]&\C\ar[r]&\F\ar[r]&\X\ar[r]&0}
\end{align}
where $\mathbf{F} \in \CPFA$ and $\mathbf{C}\in \mathbf{C}({\rm
dg}$-${\rm Cof}\CA)$. In addition, by Proposition \ref{cotor1},
$\mathbf{C}$ is a complex of cotorsion objects. Then,  \eqref{connn}
is degree-wise split and so there is a canonical morphism $ u:
\mathbf{X} \rightarrow \Sigma \mathbf{C}$ such that
$\xymatrix@C-.7pc{
\mathbf{C}\ar[r]&\mathbf{F}\ar[r]&\mathbf{X}\ar[r]^u&\Sigma
\mathbf{C}}$ is a distinguished triangle in $\KFA$. But,
$\Ext^1_{\C(\CA)}(\F,\X)=0$ for all flat complex. Therefore,
$\X\in\KPFA^\perp$ and  hence, $\KPFA\lrt\KFA$ has a right adjoint.
This implies  the equivalence $\KPFA^\perp=\mathbb{K}({\rm
dg}$-${\rm Cof}\CA)\lrt\DPFA$ of triangulated categories.

In the second part, we show that $\KPFA^\perp$=$\mathbb{K}({\rm
dg}$-${\rm Cof}\CA)$. Let $\X$ be a complex of flats in $\CA$. If
$\X\in\mathbb{K}({\rm dg}$-${\rm Cof}\CA)$, then $\X$ is isomorphic
in $\KFA$ to a dg-cotorsion complex $\X'$ of flats. But,
$\Ext^i_{\C(\CA)}(\F,\X')=0$ for all flat complex. Then, for any
flat complex $\F\in\KPFA^\perp$,
$\mathrm{Hom}_{\K(\CA)}(\F,\X)=\mathrm{Hom}_{\K(\CA)}(\F,\X')=0$.
Then $\X\in\KPFA^\perp$. Conversely, assume that $\X\in\KPFA^\perp$.
By part one, we a distinguished triangle $\xymatrix@C-.7pc{
\mathbf{X}\ar[r]&\mathbf{C}\ar[r]&\mathbf{F}\ar[r]^u&\Sigma
\mathbf{X}}$ where $\C$ is a dg-cotorsion flat complex and $\F$ is a
flat complex.  Applying $\mathrm{Hom}_{\K(\CA)}(\textmd{-},\C)$ on
the triangle and deduce the desired result.
\end{proof}

In the reminder of this section, we give some  examples of
Grothendieck categories which can be replaced by $\CG$.
\subsection{Category of $R$-modules}
Let $R$ be an associative ring with $1\neq0$ and $\CG$ be the
category of all left $R$-modules. Then $\CA=\mathrm{Fun}(\CC, \CG)$
is a generalization of the category $\C(R)$ of complexes of
$R$-modules.

\subsection{Category of quasi-coherent sheaves of $\CO_X$-modules}

Let $(X,\CO_X)$ be a quasi-compact and semi-separated scheme and
$\mathfrak{Mod}X$ be the category of all sheaves of $\CO_X$-modules
and $\mathfrak{Qco}X$ be the category of all quasi-coherent sheaves
of $\CO_X$-modules. We know that $\mathfrak{Qco}X$ is the full
subcategory of $\mathfrak{Mod}X$ and the inclusion functor
$i:\mathfrak{Qco}X\to\mathfrak{Mod}X$ has a right adjoint
$Q:\mathfrak{Mod}X\to\mathfrak{Qco}X$ (\cite{TT}). So, $Q$ preserves
injective cogenerators and by \cite[pp. 1109]{MS11}
$\mathfrak{Qco}X$ admits arbitrary products. Also, for each pair
$\CF$ and $\CF'$ of quasi-coherent sheaves of $\CO_X$-modules,
${\mathcal{H}}om_{\mathrm{q}}(\CF,\CF')=Q{\mathcal{H}}om(\CF,\CF')$
is an internal Hom structure on $\mathfrak{Qco}X$. If $\CJ$ is an
injective cogenerator in $\mathfrak{Qco}X$ and
$(\textmd{-})^+={\mathcal{H}}om_{\mathrm{q}}(\textmd{-},\CJ)$ then,
by \cite{Ho17a}, an exact sequence $\xymatrix@C-0.7pc{\mathcal{L}:
0\ar[r]&\CF'\ar[r]^f&\CF\ar[r]^g&\CF''\ar[r]&0}$ in
$\mathfrak{Qco}X$ is pure if and only if $\mathcal{L}^+$ splits.
Therefore, a quasi-coherent $\CO_X$-module $\CF$  is flat if and
only if $\CF^+$ is injective. Consequently,  flat quasi-coherent
$\CO_X$-modules are flat in the sense of Definition \ref{safi11}. By
\cite[\S 3]{Ho17a},  $\mathfrak{Qco}X$ does  not have enough
projective objects. But, any quasi-coherent sheaves of
$\CO_X$-module is a quotient of a flat quasi-coherent sheaf (see
\cite[Corollary 3.21]{Mu07}).

\end{document}